\documentclass[11pt,oneside,leqno]{amsart}
\usepackage{amsmath}
\usepackage{amsfonts}
\usepackage{amssymb}
\usepackage{color}
\usepackage{graphicx}
\newtheorem{theorem}{Theorem}[section]
\newtheorem*{theorem*}{Theorem}
\newtheorem{definition}[theorem]{Definition}

\newtheorem{proposition}[theorem]{Proposition}

  \newcommand\g{{\mathfrak{g}}}
\newcommand\h{{\mathfrak{h}}}
\newcommand\m{{\mathfrak{m}}}

\begin{document}

\title{Cyclic Lorentzian Lie Groups}
\author[G.~Calvaruso]{Giovanni Calvaruso}
\address{Dipartimento di Matematica e Fisica \lq\lq E. De Giorgi\rq\rq \\
Universit\`a del Salento \\ Prov. Lecce-Arnesano \\
73100,  Lecce \\ Italy.}
\email{giovanni.calvaruso@unisalento.it}
\thanks{Giovanni Calvaruso has been partially supported by funds of the University of Salento and MURST (PRIN). Both authors have been partially supported by MINECO (Spain) under grant
MTM2014-53201-P}

\author[M.~Castrill\'{o}n L\'{o}pez]{M. Castrill\'{o}n L\'{o}pez}
\address{ICMAT (CSIC-UAM-UC3M-UCM)\\
Departmento de Geometr\'\i a y Topolog\'\i a \\
Universidad Complutense de Madrid \\ 28040 Madrid \\ Spain.}
\email{mcastri@mat.ucm.es}

\date{}

\subjclass[2010]{53C30, 53C50, 22E25, 22E46.}
\keywords{Lorentzian Lie groups, left-invariant cyclic metrics, homogeneous pseudo-Riemannian structures.}

\begin{abstract}
We consider Lie groups equipped with a left-invariant cyclic
Lorentzian metric. As in the Riemannian case, in terms of
homogeneous structures, such metrics can be considered as
different as possible from bi-invariant metrics. We show that
several results concerning cyclic Riemannian metrics do not extend
to their Lorentzian analogues, and obtain a full classification of
three- and four-dimensional cyclic Lorentzian metrics.
\end{abstract}

\maketitle

\section{Introduction}

Homogeneous Riemannian manifolds were characterized in terms of
homogeneous structures by Ambrose and Singer \cite{AS} (see also
\cite{TV}). Gadea and Oubi\~na \cite{GO} introduced {\em
homogeneous pseudo-Riemannian structures}, to give a corresponding
characterization of reductive homogeneous pseudo-Riemannian
manifolds. Let $G$ denote a (connected) Lie group and $\g$ its Lie
algebra. It is well known that left-invariant pseudo-Riemannian
metrics $g$ on $G$ are in a one-to-one correspondence with
nondegenerate inner products on $\g$, which we shall denote again
by $g$. If $g$ is such an inner product on $\g$ and $\nabla$
denotes its Levi-Civita connection, then tensor $S_x y=\nabla_x y,
\ x,y\in \g,$ is a homogeneous pseudo-Riemannian structure.
Conversely, among homogeneous pseudo-Riemannian manifolds,
pseudo-Riemannian Lie groups are characterized by the fact that
they admit a global pseudo-orthonormal frame field $\{e_i \}$,
such that $S_{e_i} e_j=\nabla _{e_i} e_j$ defines a homogeneous
pseudo-Riemannian structure (see for example \cite{C1}).

A systematic study of left-invariant Riemannian cyclic metrics
started in \cite{GGO}, with particular regard to the semi-simple
and solvable cases and a complete classification of the examples
of dimension up to five. Following \cite{GGO}, a left-invariant
pseudo-Riemannian metric $g$ is said to be {\em cyclic} if the
homogeneous pseudo-Riemannian structure $S$ described above falls
within $\mathcal{S}_1 \oplus \mathcal{S}_2$ in Tricerri-Vanhecke's
classification of homogeneous structures. Explicitly, this means
that
\begin{equation}\label{cyclic}
\mathfrak{S}_{x,y,z} g([x,y],z)=0 \quad \text{for all} \; x,y,z
\in \g ,
\end{equation}
where $\mathfrak{S}$ stands for the cyclic sum. Note that, as
bi-invariant metrics are characterized by condition $S \in
\mathcal{S}_3$, cyclic metric can be considered as different as
possible from the bi-invariant ones.

In this paper, we undertake the investigation of left-invariant
cyclic pseudo-Riemannian metrics, starting from the Lorentzian
ones. Although four-dimensional connected, simply connected
Lorentzian Lie groups coincide with the Riemannian ones, their
geometry proves to be richer, also with regard to cyclic metrics.
We shall classify cyclic Lorentzian Lie groups of dimension up to
four and show that several rigidity results valid for Riemannian
cyclic metrics do not extend to pseudo-Riemannian settings. In
particular, differently from the Riemannian case, we show the
existence of compact or nilpotent non-abelian cyclic Lorentzian
Lie groups.

The paper is organized in the following way. In Section~2 we shall
report some basic information concerning homogeneous structures
and  cyclic metrics. In Sections~3 and 4 we shall give the
complete classification of left-invariant cyclic Lorentzian
metrics in dimension three and four, respectively. In particular,
Theorems~\ref{cyclic3D}, \ref{hRie}, \ref{hLore} and \ref{hDeg}
below show that contrarily to the Riemannian case, all possible
connected and simply connected three- and  four-dimensional Lie
groups admit an appropriately chosen left-invariant Lorentzian
cyclic metric. We conclude in Section~5 with the classification of
cotorsionless Lorentzian three-manifolds, and some observations,
concerning in particular the link between three- and
four-dimensional cyclic Lie groups, and the obstruction to the
construction of non-symmetric solvmanifolds from solvable cyclic
groups.

\section{Preliminaries}
\setcounter{equation}{0}


Let $M$ be a connected manifold and $g$ a pseudo-Riemannian metric
on $M$. We denote by $\nabla$ the Levi-Civita connection of
$(M,g)$ and by $R$ its curvature tensor. The following definition
was introduced by Gadea and Oubi\~na:

\begin{definition}\label{dd}\cite{GO}
A homogeneous pseudo-Riemannian structure  {\em on $(M,g)$ is a
tensor field $S$ of type $(1,2)$ on $M$, such that the connection
$\tilde{\nabla} = \nabla -S$ satisfies}
$$ 
\tilde{\nabla} g=0, \qquad \tilde{\nabla} R=0, \qquad \tilde{\nabla} S =0.
$$ 
\end{definition}

\noindent The geometric meaning of the existence of a homogeneous
pseudo-Rieman\-nian structure is explained by the following
result.

 \begin{theorem}\label{GO}\cite{GO}
Let $(M,g)$ be a connected, simply connected and complete
pseudo-Riemannian manifold. Then, $(M,g)$ admits a
pseudo-Riemannian structure if and only if it is a reductive
homogeneous pseudo-Riemannian manifold.
 \end{theorem}

\noindent Observe that if any of the hypotheses of connectedness,
simple connectedness or completeness is missing, the existence of
a homogeneous structure characterizes local homogeneity of the
manifold. We remark that, while any homogeneous Riemannian
manifold is reductive, a homogeneous pseudo-Riemannian manifold
needs not be reductive. This restriction also happens when
considering local homogeneity, although a precise definition of
local reductivity is required in this context (see \cite{Lu}).
Definition~\ref{dd} and Theorem~{\ref{GO} above extend the
characterization of homogeneous Riemannian manifolds by means of
homogeneous structures \cite{AS} to {\em reductive} homogeneous
pseudo-Riemannian manifolds.

We explicitly recall that for the reductive homogeneous
pseudo-Rieman\-nian manifold $(M=G/H,g)$, with reductive
decomposition $\g=\h \oplus \m$,  the linear connection
$\tilde\nabla=\nabla-T$ is the canonical connection associated to
the reductive decomposition \cite{TV}.

Let $V$ denote an $n$-dimensional real vector space, equipped with
a non-degenerate inner product $\langle , \rangle$ of signature
$(k,n-k)$. It is the model space for the tangent space at each
point of a homogeneous pseudo-Riemannian manifold $(M,g)$. Let
$\mathcal{S} (V)$ denote the vector space of $(0,3)$-tensors $S$
on $V$, satisfying the same condition as the first equation
$\tilde{\nabla}g=0$ of a homogenous structure, that is,
$$
\mathcal{S} (V) = \left\{S \in \bigotimes ^3 V^* :
S_{xyz}=-S_{xzy}, \; x,y,z \in V \right\},
$$
where $S_{xyz}:=\langle S_x y,z\rangle$. Then, $\langle,\rangle$
induces an inner product on  $\mathcal{S} (V)$, given by
$$
\langle S, S'\rangle= \sum_{i,j,k=1}^n \varepsilon_i \varepsilon_j
\varepsilon_k S_{e_i e_j e_k} S'_{e_i e_j e_k},
$$
where $\{e_i \}$ denotes a pseudo-orthonormal basis of $V$ and
$\varepsilon_i=\langle e_i,e_i\rangle$  for all indices $i$. The
following result was proved \cite{GO2}.

\begin{theorem}\label{S-i}\cite{GO2}
If $\dim V \geq 3$, then $\mathcal{S} (V)$ decomposes into the orthogonal direct sum
$$\mathcal{S} (V)=\mathcal{S}_1 (V)\oplus\mathcal{S}_2 (V)\oplus\mathcal{S}_3 (V),$$
where
$$\begin{array}{l}
\mathcal{S}_1 (V)=\left\{S \in \mathcal{S}(V): S_{xyz}=\langle x,y\rangle \omega(z)-\langle x,z\rangle \omega(y), \; \omega \in V^*  \right\}, \\[6pt]
\mathcal{S}_2 (V)=\left\{S \in \mathcal{S} (V): \mathfrak{S}_{xyz} S_{xyz}=0, \; c_{12}(S):=\sum_{i=1}^n \varepsilon_i S_{e_i e_i \cdot}=0  \right\}, \\[6pt]
\mathcal{S}_3 (V)=\left\{S \in \mathcal{S} (V): S_{xyz}+S_{yxz}=0 \right\}
\end{array}$$
are invariant and irreducible under the action of $O(k,n-k)$. If $\dim V=2$, then $\mathcal{S}(V)=\mathcal{S}_1(V)$. Furthermore,
$$\begin{array}{l}
\mathcal{S}_1 (V) \oplus \mathcal{S}_2 (V)=\left\{S \in \mathcal{S}(V): \mathfrak{S}_{xyz} S_{xyz}=0  \right\}, \\[6pt]
\mathcal{S}_2 (V) \oplus \mathcal{S}_3 (V)=\left\{S \in \mathcal{S} (V):  c_{12}(S)=0  \right\}, \\[6pt]
\mathcal{S}_1 (V)\oplus \mathcal{S}_3 (V)=\left\{S \in \mathcal{S}
(V): \begin{array}{l}S_{xyz}+S_{yxz}=2\langle x,y\rangle \omega(z)
\\-\langle x,z\rangle \omega(y)-\langle y,z\rangle \omega(x)
\end{array} ,\; \omega \in V^*  \right\} .
\end{array}$$
\end{theorem}

As proved in \cite{GO2}, {\em naturally reductive} homogeneous
pseudo-Riemannian manifolds are all and the ones admitting a
homogeneous structure $S \in \mathcal{S}_3 (V)$, while {\em
cotorsionless manifolds} are characterized by the existence of
homogeneous structures $S \in \mathcal{S}_1 (V) \oplus
\mathcal{S}_2 (V)$.

Among homogeneous pseudo-Riemannian manifolds, pseudo-Riemannian
Lie groups are characterized by the existence of a special
homogeneous pseudo-Riemannian structure (see also \cite{C1}). In
fact, when $(M=G,g)$ is a Lie group equipped with a left-invariant
Lorentzian metric $g$, uniquely determined at the algebraic level
by a non-degenerate inner product $g$ on the Lie algebra $\g$,
tensor  $S_x y=\nabla_x y, \ x,y\in \g,$ defines a homogeneous
pseudo-Riemannian structure. In this case $\tilde\nabla$, which
vanishes when evaluated on left invariant vector fields, is the
so-called $(-)$-{\em connection of Cartan-Schouten}, whose
curvature and torsion are respectively given by $\tilde R=0$ and
$\tilde T (X,Y)=-[X,Y]$.

It is well known that the left-invariant pseudo-Riemannian metric
corresponding to $g$ is {\em bi-invariant} if and only if the
above special homogeneous structure $S$ belongs to $\mathcal{S}_3
(V)$. On the other hand, $g$ is called {\em cyclic} when $S \in
\mathcal{S}_1 (V) \oplus \mathcal{S}_2 (V)$. Thus, taking into
account the orthogonal decomposition of $\mathcal{S}(V)$,
left-invariant cyclic metrics can be considered \lq\lq as far away
as possible\rq\rq \ from the bi-invariant ones.

We report below several strong rigidity results obtained in
\cite{GGO} for {\em Riemannian} cyclic metrics. As a consequence
of the classifications given in the next sections, we shall see
that most of these result do not hold any more for Lorentzian
cyclic metrics.

\begin{proposition}\cite{GGO}\label{p1}
A connected cyclic Riemannian Lie group is flat if and only if it
is abelian. Moreover, let $G$ be a non-abelian cyclic Riemannian
Lie group.
\begin{itemize}
\item[(i)] If $G$ is solvable, then it has strictly negative
scalar curvature. \item[(ii)] If $G$ is unimodular, then it has
positive sectional curvatures. If moreover it is solvable, then it
has both positive and negative curvatures. \item[(iii)] If $G$ is
not unimodular there exist negative sectional curvatures.
\end{itemize}
\end{proposition}

\begin{theorem}\cite{GGO}\label{p2}
Every non-abelian cyclic Riemannian Lie group is not compact.
\end{theorem}

\begin{theorem}\cite{GGO}\label{p3}
The universal covering $\widetilde{SL}(2,\mathbb R)$ of
$SL(2,\mathbb R)$ is the only connected, simply connected simple real
Riemannian Lie group.
\end{theorem}

\begin{proposition}\cite{GGO}\label{p4}
Non-abelian nilpotent Lie groups do not admit left-in\-variant
Riemannian cyclic metrics.
\end{proposition}


We end this section clarifying the relationship between Riemannian
and Lorentzian Lie groups.
Let $G$ be an $n$-dimensional connected Lie group and $\g$ its Lie algebra.
Left-invariant Lorentzian metrics on $G$ are in a one-to-one correspondence
with inner products on $\g$ of signature $(n-1,1)$. If $g$ is such a
Lorentzian inner product, then it exists a pseudo-orthonormal
basis $\{e_1,\dots,e_n\}$ of $\g$, with $e_n$ time-like. But then, $G$ also
admits a corresponding left-invariant Riemannian metric, completely determined
at the Lie algebra level by having $\{e_1,\dots,e_n\}$ as an orthonormal basis
of $\g$.

Conversely, given a positive definite inner product $\bar g$ over
$\g$, and  a $\bar g$-orthonormal basis $\{e_1,\dots,e_n\}$ of
$\g$, it suffices to change the causal character of one of vectors
in the basis, choosing it to be time-like, to determine a
left-invariant Lorentzian metric on $G$. Therefore, the following
result holds (see also \cite{CZ}).

\begin{proposition}\label{RieLo}
The class of $n$-dimensional connected, simply connected Lorentzian Lie
groups (respectively, Lorentzian Lie algebras) coincides with the
class of the Riemannian ones.
\end{proposition}

We explicitly observe that, although connected, simply connected Lorentz\-ian
Lie groups coincide with the Riemannian ones
(Proposition~\ref{RieLo}), the geometry of left-invariant
Lorentzian metrics is much richer than the one of their Riemannian
counterpart. The fundamental reason for such a difference is the
existence in Lorentzian settings of vectors with different causal
characters. Some consequences of this fact are:
\begin{itemize}
\item[$\bullet$]  that (contrarily to the Riemannian case) a
self-adjoint operator with respect to a Lorentzian metric needs
not be diagonalizable. For example, this yields four standard
forms of three-dimensional unimodular Lorentzian Lie groups
\cite{R}, while just one form occurs in Riemannian settings
\cite{M};

\item[$\bullet$]  that every subspace of a vector space endowed
with a positive definite inner product, inherits a positive inner
product, while a subspace of a Lorentzian vector space inherits an
inner product that can be either positive definite, Lorentzian, or
even degenerate. In particular, this fact yields the differences
in the classifications of three-dimensional non-unimodular
Lorentzian \cite{CP} and Riemannian \cite{M} Lie groups, and of
left-invariant Lorentzian \cite{CZ} and Riemannian metrics
\cite{AK} on four-dimensional Lie groups.
\end{itemize}

\section{Three-dimensional cyclic Lorentzian Lie groups}
\setcounter{equation}{0}

As proved in \cite{GO2} and reported in the above
Theorem~\ref{S-i}, for a two-dim\-en\-sional vector space $V$, one
has $\mathcal{S}(V)=\mathcal{S}_1(V)$. Consequently, any
two-dimensional pseudo-Riemannian Lie group is cyclic. Next,
homogeneous Lorentzian three-manifolds were classified in
\cite{C1}, taking into account previous results of Rahmani
\cite{R} and Cordero and Parker \cite{CP}. The classification
result is the following.

\begin{theorem}{\bf \cite{C1}}\label{Thom}
A three-dimensional connected, simply connected complete homogeneous
Lorentzian manifold $(M,g)$ is either symmetric, or  $M=G$ is a
three-dimensional Lie group and $g$ is left-invariant. Precisely,
one of the following cases occurs:

\medskip
I) If $G$ is unimodular, then there exists a pseudo-orthonormal
frame field $\{ e_1,e_2,e_3\}$, with $e_3$ time-like, such that
the Lie algebra of $G$ is one of the following:

\begin{eqnarray}
& &\left[e_1,e_2 \right]=\alpha e_1-\beta e_3, \nonumber \\
\mathfrak{g} _1  : & &\left[ e_1,e_3\right]=-\alpha e_1-\beta e_2, \label{g1}\\
& & \left[e_2,e_3\right]=\beta e_1 +\alpha e_2 +\alpha  e_3 \qquad \alpha \neq 0.  \nonumber
\end{eqnarray}

If $\beta \neq 0$, then $G$ is $\widetilde{SL}(2,\mathbb R)$,
while for $\beta=0$, $G=E(1,1)$ is the group of rigid motions of
the Minkowski two-space.

\begin{eqnarray}
& &\left[e_1,e_2 \right]=-\gamma e_2-\beta e_3, \nonumber \\
\mathfrak{g} _2: & &\left[ e_1,e_3\right]=-\beta e_2+\gamma e_3, \qquad \gamma \neq 0, \label{g2} \\
& & \left[e_2,e_3\right]=\alpha e_1 .  \nonumber
\end{eqnarray}
In this case, $G=\widetilde{SL}(2,\mathbb R)$ if $\alpha \neq 0$,
while $G=E(1,1)$ if $\alpha=0$.

\begin{equation}\label{g3}
(\mathfrak{g} _3): \quad \left[e_1,e_2 \right]=-\gamma e_3, \quad \left[ e_1,e_3\right]=-\beta e_2, \quad
 \left[e_2,e_3\right]=\alpha e_1 .
\end{equation}
The following Table~I (where $\widetilde{E}(2)$ and $H_3$
respectively denote the universal covering of the group of rigid
motions in the Euclidean two-space and the Heisenberg group) lists
all the Lie groups $G$ which admit a Lie algebra {\bf $g_3$},
according to the different possibilities for $\alpha$, $\beta$ and
$\gamma$:
\begin{center}
\begin{tabular}{|c|c|c|c|}
\hline Lie group &  $\alpha$ & $\beta$ & $\gamma$ $\vphantom{\displaystyle{A^{B^{C}}}}$\\
\hline  $\widetilde{SL}(2,\mathbb R)$ & $+$ & $+$ & $+$ $\vphantom{\displaystyle{A^{B^{C^{D}}}}}$\\
\hline  $\widetilde{SL}(2,\mathbb R)$ & $+$ & $-$ & $-$ $\vphantom{\displaystyle{A^{B^{C^{D}}}}}$ \\
\hline  $SU(2)$ & $+$ & $+$ & $-$ $\vphantom{\displaystyle{A^{B^{C}}}}$\\
\hline  $\widetilde{E}(2)$ & $+$ & $+$ & $0$ $\vphantom{\displaystyle{A^{B^{C^{D}}}}}$ \\
\hline  $\widetilde{E}(2)$ & $+$ & $0$ & $-$ $\vphantom{\displaystyle{A^{B^{C^{D}}}}}$ \\
\hline  $E(1,1)$ & $+$ & $-$ & $0$ $\vphantom{\displaystyle{A^{B^{C}}}}$\\
\hline  $E(1,1)$ & $+$ & $0$ & $+$ $\vphantom{\displaystyle{A^{B^{C}}}}$\\
\hline  $H_3$ & $+$ & $0$ & $0$ $\vphantom{\displaystyle{A^{B^{C}}}}$\\
\hline  $H_3$ & $0$ & $0$ & $-$ $\vphantom{\displaystyle{A^{B^{C}}}}$\\
\hline  $\mathbb R \oplus \mathbb R \oplus \mathbb R$ & $0$ & $0$ & $0$ $\vphantom{\displaystyle{A^{B^{C}}}}$\\
\hline
\end{tabular} \nopagebreak \\ \nopagebreak {\em Table~I: 3D Lorentzian Lie groups with Lie algebra $\g_3$} $\vphantom{\displaystyle\frac{a}{2}}$
\end{center}

\begin{eqnarray}
& &\left[e_1,e_2 \right]=- e_2 + (2 \varepsilon - \beta) e_3, \qquad \varepsilon = \pm 1, \nonumber \\
\mathfrak{g} _4: & &\left[ e_1,e_3\right]=-\beta e_2 + e_3,  \label{g4} \\
& & \left[e_2,e_3\right]=\alpha e_1 .  \nonumber
\end{eqnarray}
Table~II below describes all Lie groups $G$ admitting a Lie algebra {\bf $g_4$}:
\begin{gather*}
\begin{array}{cc}
\begin{tabular}{|c|c|c|}
\hline Lie group \quad  {\rm ($\varepsilon =1$)} &  $\alpha$ & $\beta$  $\vphantom{\displaystyle{A^{B^{C}}}}$ \\
\hline  $\widetilde{SL}(2,\mathbb R)$ & $\neq 0$ & $\neq 1$ $\vphantom{\displaystyle{A^{B^{C^{D}}}}}$ \\
\hline  $E(1,1)$ & $0$ & $\neq 1$ $\vphantom{\displaystyle{A^{B^{C}}}}$ \\
\hline  $E(1,1)$ & $<0$ & $1$  $\vphantom{\displaystyle{A^{B^{C}}}}$ \\
\hline  $\widetilde{E}(2)$ & $>0$ & $1$ $\vphantom{\displaystyle{A^{B^{C^{D}}}}}$ \\
\hline  $H_3$ & $0$ & $1$ $\vphantom{\displaystyle{A^{B^{C}}}}$ \\
\hline
\end{tabular}
& \hspace{-2mm}
\begin{tabular}{|c|c|c|}
\hline Lie group \quad {\rm ($\varepsilon =-1$)} &  $\alpha$ & $\beta$  $\vphantom{\displaystyle{A^{B^{C}}}}$ \\
\hline  $\widetilde{SL}(2,\mathbb R)$ & $\neq 0$ & $\neq -1$ $\vphantom{\displaystyle{A^{B^{C^{D}}}}}$ \\
\hline  $E(1,1)$ & $0$ & $\neq -1$ $\vphantom{\displaystyle{A^{B^{C}}}}$ \\
\hline  $E(1,1)$ & $>0$ & $-1$ $\vphantom{\displaystyle{A^{B^{C}}}}$ \\
\hline  $\widetilde{E}(2)$ & $<0$ & $-1$ $\vphantom{\displaystyle{A^{B^{C^{D}}}}}$ \\
\hline  $H_3$ & $0$ & $-1$ $\vphantom{\displaystyle{A^{B^{C}}}}$\\
\hline
\end{tabular}
\end{array}
 \\
 \text{{\em Table~II: 3D Lorentzian Lie groups with Lie algebra $\g_4$}}
 \vphantom{\displaystyle\frac{a}{2}}
 \end{gather*}
%
II) If $G$ is non-unimodular, then there exists a
pseudo-orthonormal frame field $\{ e_1,e_2,e_3\}$, with $e_3$
time-like, such that the Lie algebra of $G$ is one of the
following:
\begin{eqnarray}
& &\left[e_1,e_2 \right]=0, \nonumber \\
\mathfrak{g} _5: & &\left[ e_1,e_3\right]=\alpha e_1+\beta e_2, \label{g5} \\
& & \left[e_2,e_3\right]=\gamma e_1 +\delta e_2, \qquad \alpha +\delta \neq 0, \,
\alpha \gamma +\beta \delta =0.  \nonumber
\end{eqnarray}
\begin{eqnarray}
& &\left[e_1,e_2 \right]=\alpha e_2 +\beta e_3, \nonumber \\
\mathfrak{g} _6: & &\left[ e_1,e_3\right]=\gamma e_2+\delta e_3, \label{g6} \\
& & \left[e_2,e_3\right]= 0, \qquad \qquad \qquad \alpha +\delta \neq 0, \, \alpha \gamma -\beta \delta =0.  \nonumber
\end{eqnarray}
\begin{eqnarray}
& &\left[e_1,e_2 \right]=-\alpha e_1-\beta e_2 -\beta e_3, \nonumber \\
\mathfrak{g} _7: & &\left[ e_1,e_3\right]=\alpha e_1+\beta e_2 +\beta e_3, \label{g7} \\
& & \left[e_2,e_3\right]=\gamma e_1 +\delta e_2 +\delta  e_3 ,  \qquad \alpha +\delta \neq 0, \, \alpha \gamma =0. \nonumber
\end{eqnarray}
\end{theorem}

\noindent With the obvious exception of $\mathbb S^2 \times
\mathbb R$, every three-dimensional Lorentzian symmetric space can
also be realized in terms of a suitable Lorentzian Lie group
\cite[Theorem~4.2]{C2}. Hence, apart from $\mathbb S^2 \times
\mathbb R$, the classification of three-dimensional Lorentzian
cotorsionless manifolds reduces to the one of three-dimensional
Lorentzian Lie groups.

In order to have a cyclic metric $g$, it suffices to check
condition \eqref{cyclic} on the vectors of a basis $\{e_i\}$ of $\g$, that is,
\begin{equation}\label{cyclicbase}
\mathfrak{S}_{i,j,k=1}^3 \ g([e_i,e_j],e_k)=0 \quad \text{for all
indices} \; i,j,k.
\end{equation}
Note that if two of indices $i,j,k$ coincide, then
equation~\eqref{cyclicbase} is trivially satisfied. Hence, in the
three-dimensional case, $g$ is cyclic if and only if
\begin{equation}\label{cyclicbase3D}
 g([e_1,e_2],e_3)+g([e_2,e_3],e_1)+g([e_3,e_1],e_2)=0.
\end{equation}
For each three-dimensional Lorentzian Lie group, the above Theorem~\ref{Thom} provides an explicit description of the corresponding Lie algebra in terms of a pseudo-orthonormal basis $\{e_1,e_2,e_3\}$ of $\g$, with $e_3$ time-like.  We now check equation~\eqref{cyclicbase3D} for these examples and we get the following cases:
\begin{itemize}
\item[1)] $\mathfrak{g} _1$ is cyclic if and only if $\beta=0$;
\item[2)] $\mathfrak{g} _2$ is cyclic if and only if $\alpha=-2\beta$;
\item[3)] $\mathfrak{g} _3$ is cyclic if and only if $\alpha+\beta+\gamma=0$;
\item[4)] $\mathfrak{g} _4$ is cyclic if and only if $\alpha=2(\varepsilon-\beta)$;
\item[5)] $\mathfrak{g} _5$ is cyclic if and only if $\beta-\gamma=0$;
\item[6)] $\mathfrak{g} _6$ is cyclic if and only if $\beta+\gamma=0$;
\item[7)] $\mathfrak{g} _7$ is cyclic if and only if $\gamma=0$.
\end{itemize}
Therefore, taking into account the above Theorem~\ref{Thom}, we
proved the following result.

\begin{theorem}\label{cyclic3D}
A three-dimensional connected, simply connected non-abelian cyclic Lorentzian Lie
group is isometrically isomorphic to one of the following Lie
groups:

\smallskip
I) In the unimodular case:
\begin{itemize}
\item[(a)] $E(1,1)$, with Lie algebra described by one of the
following cases: \newline $\g_1$ with $\beta=0$; $\g_2$ with
$\alpha=\beta=0$; $\g_3$ with $\alpha+\beta=\gamma=0$;
\vspace{4pt}\item[(b)] $\widetilde{SL}(2,\mathbb R)$, with Lie
algebra described by one of the following cases: \newline $\g_2$
with $\alpha=-2\beta \neq 0$; $\g_3$ with
$\alpha=-(\beta+\gamma)>0$ and $\beta,\gamma <0$; $\g_4$ with
$\alpha=2(\varepsilon-\beta)\neq 0$; \vspace{4pt}\item[(c)]
$SU(2)$, with Lie algebra described by $\g_3$ with
$\alpha=-(\beta+\gamma)$ and $\gamma <0<\beta$;
\vspace{4pt}\item[(d)] $\tilde E(2)$, with Lie algebra described
by $\g_3$ with $\beta=\alpha+\gamma=0$ and $\gamma <0$;
\vspace{0.5pt}\item[(e)] $H_3$, with Lie algebra described by
$\g_4$ with $\alpha=\varepsilon-\beta=0$.
\end{itemize}

\smallskip
II) In the non-unimodular case: the connected, simply connected Lie group
$G$, whose Lie algebra is either  $\g_5$ with $\beta=\gamma$,
$\g_6$ with $\beta=-\gamma$, or $\g_7$ with $\gamma=0$.
\end{theorem}

\noindent Note that in general, each of the cases listed in the
above Theorem~\ref{cyclic3D} gives rise to a family of
left-invariant cyclic Lorentzian metrics, depending on one or more
parameters.

Curvature properties of three-dimensional Lorentzian Lie groups
have been determined in \cite{C2}. Together with the examples
classified in Theorem~\ref{cyclic3D}, the results of \cite{C2}
already permit to emphasize some deep differences among Lorentzian
and Riemannian cyclic metrics. In fact:

\begin{itemize}
\item[(1)] $SU(2)$ is a connected, simply connected Lie group, both {\em
compact} and {\em simple}. Hence, case (c) of
Theorem~\ref{cyclic3D} yields a Lorentzian counterexample to both
Theorem~\ref{p2} and Theorem~\ref{p3}. \item[(2)] The Heisenberg
group $H_3$ is non-abelian and nilpotent. Hence, case (e) of
Theorem~\ref{cyclic3D} yields a Lorentzian counterexample to both
Proposition~\ref{p1} and Proposition~\ref{p4}. \item[(3)]
Non-unimodular Lie group $G$, with Lie algebra $\g_7$ satisfying
either $\alpha=\gamma=0$ or $\gamma=0\neq \alpha=\delta$, is
equipped with a flat cyclic Lorentzian metric, giving a Lorentzian
counterexample to Proposition~\ref{p1},(iii).
\end{itemize}

\section{Four-dimensional cyclic Lorentzian Lie groups}
\setcounter{equation}{0}

As we observed in Section~2 (Proposition~\ref{RieLo}), in any
dimension $n$,  connected, simply connected Lorentzian Lie groups coincide
with the Riemannian ones. Taking into account the classification
of four-dimensional Rieman\-nian Lie groups given by
B\'erard-B\'ergery in  \cite{BB}, we then have the following.

\begin{proposition}\label{4DLG}
The connected and simply connected four-dimensional Lorentian Lie
groups are:
\begin{itemize}
\item[(i)] the (unsolvable) direct products $SU(2) \times \mathbb R$ and $\widetilde{SL}(2,\mathbb R) \times \mathbb R$;
\item[(ii)] one of the following solvable Lie groups:
\begin{itemize}
\item[(ii1)] the non-trivial semi-direct products $\tilde{E}(2)
\rtimes \mathbb R$ and $E(1,1) \rtimes \mathbb R$; \item[(ii2)]
the non-nilpotent semi-direct products $H_3 \rtimes \mathbb R$
($H_3$ denoting the Heisenberg group); \item[(ii3)] the
semi-direct products $\mathbb R^3 \rtimes \mathbb R$.
\end{itemize}
\end{itemize}
\end{proposition}

We observe that all the examples classified in the above
Proposition share the same fundamental structure, in the sense
that all their Lie algebras $\g$ are of the form $\g=\mathfrak{h}
\rtimes \mathfrak{r}$, where $\mathfrak{r}$ is a one-dimensional
Lie algebra, spanned by a vector acting (possibly in a trivial
way) as a derivation on a  three-dimensional unimodular Lie
algebra $\h$.

Semi-direct products involving a three-dimensional non-unimodular
Lie algebra do not appear in the above classification. Indeed, it
is easy to check that a semi-direct product $\tilde\h \rtimes
\mathfrak{r}$, with $\tilde\h$ non-unimodular, is also isomorphic
to a semi-direct product $\h \rtimes \tilde{\mathfrak{r}}$, with
$\h$ unimodular.

To make the Lorentzian case more interesting than its Riemannian
counterpart, we have the following fundamental difference: if $g$
is a positive definite inner product on $\g=\mathfrak{h} \rtimes
\mathfrak{r}$, the same is true for its restriction $g|_{\h}$ over
$\mathfrak{h}$. However, if $g$ is Lorentzian, then three
different cases can occur, as $g|_{\h}$ is either
\begin{itemize}
\item[(a)] {\em positive definite}, \ (b) {\em Lorentzian}, or \
(c) {\em degenerate}.
\end{itemize}
We now give the following key result.

\begin{proposition}\cite{CZ}\label{gandg}
Let $(\g,g)$ be an arbitrary four-dimensional Lorentzian Lie
algebra. Then, there exists a basis $\{e_1,e_2,e_3,e_4 \}$ of
$\g$, such that
\begin{itemize}
\item $\h={\rm span}(e_1,e_2,e_3)$ is a three-dimensional Lie
algebra and $e_4$ acts as a derivation on $\h$ (that is,
$\g=\mathfrak{h}\rtimes \mathfrak{r}$, where $\mathfrak{r} ={\rm
span}(e_4)$), and \item with respect to $\{e_1,e_2,e_3,e_4 \}$,
the Lorentzian inner product takes one of the following forms:
$$(a) \;
\left(\begin{array}{cccc}
1 & 0 & 0 & 0 \\
0 & 1 & 0 & 0 \\
0 & 0 & 1 & 0 \\
0 & 0 & 0 & -1 \\
\end{array}\right),  \quad
(b) \;
\left(\begin{array}{cccc}
1 & 0 & 0 & 0 \\
0 & 1 & 0 & 0 \\
0 & 0 & -1 & 0 \\
0 & 0 & 0 & 1 \\
\end{array}\right),  \quad
(c) \;
\left(\begin{array}{cccc}
1 & 0 & 0 & 0 \\
0 & 1 & 0 & 0 \\
0 & 0 & 0 & 1 \\
0 & 0 & 1 & 0 \\
\end{array}\right).
$$
\end{itemize}
\end{proposition}

\begin{proof} The following argument partially corrects
and replaces the proof of Proposition~2.3 in \cite{CZ}. Consider a
semi-direct product $\g=\mathfrak{k} \rtimes \mathfrak{r}$ of two
Lie algebras $\mathfrak{r}$ and $\mathfrak{k}$, with
$\mathfrak{r}={\rm span}(v)$ one-dimensional. Note that for any
vector $w \in \mathfrak{k}$  we have again $\g=
\mathfrak{k}\rtimes \tilde{\mathfrak{r}}$, where
$\tilde{\mathfrak{r}}={\rm span}(v+w)$. In fact, since
$\mathfrak{r}$ is one-dimensional,
$\mathfrak{g}=\mathfrak{k}\rtimes \mathfrak{r}$ means that
$[\mathfrak{r},\mathfrak{r}]=0$,
$[\mathfrak{k},\mathfrak{k}]\subset \mathfrak{k}$ and
$[\mathfrak{r},\mathfrak{k}] \subset \mathfrak{k}$. From these
equations and the definition of $\tilde{\mathfrak{r}}$ it then
follows at once that the same conditions hold replacing
$\mathfrak{r}$ by $\tilde{\mathfrak{r}}$, that is,
$\g=\mathfrak{k} \rtimes \tilde{\mathfrak{r}}$.

Let $g$ denote a Lorentzian inner product on a four-dimensional
Lie algebra $\g$. Then, by the above Proposition~\ref{4DLG}, we
know that $\g=\mathfrak{h}\rtimes \mathfrak{r}$, where
$\mathfrak{r}={\rm span}(v)$ is one-dimensional. We now study
separately three cases, according on whether the restriction of
$g$ on $\h$ is respectively (a) {positive definite}, (b)
Lorentzian, or (c) degenerate.

\medskip
{\bf Case (a).} Since $g|_{\h}$ is positive definite, there exists
an orthonormal basis $\{e_1,e_2,e_3\}$ for $g|_{\h}$.

If $\mathfrak{r}={\rm span} (v)$, we now consider the orthogonal
projection $w$ of $v$ on $\h$, that is, $w:=\sum _{i=1} ^3
g(v,e_i)e_i$. Next, we put $\tilde{v} :=v-w$ and
$\tilde{\mathfrak{r}}:={\rm span}(\tilde{v})$. By the above
remark, we still have $\g=\mathfrak{h} \rtimes
\tilde{\mathfrak{r}}$.

Moreover, $\tilde{v}$ is orthogonal to $e_1,e_2,e_3$ and so,
$\tilde{\mathfrak{r}}=\mathfrak{h} ^{\perp}$. Since $g|_{\h}$ is
non-degenerate, so is $\tilde{\mathfrak{r}}=\mathfrak{h}
^{\perp}$, and the index of $g$ is the sum of the indices of
$g|_{\h}$ and $g|_{\h ^\perp}$ \cite{O'N}. Hence, $\tilde{v}$ is
necessarily time-like, and $g$ takes the form (a) with respect to
the pseudo-orthonormal basis $\{e_1,e_2,e_3,e_4\}$ of $\g$, where
we put $e_4=\tilde{v}/\sqrt{-g(\tilde{v},\tilde{v})}$.

\medskip
{\bf Case (b).} We proceed like in Case (a), with the following
slight differences: in $\h$ we now fix a pseudo-orthonormal basis
$\{e_1,e_2,e_3\}$, with $e_3$ time-like, and the orthogonal
projection $w$ of $v$ on $\h$ is given by $w:=\sum _{i=1} ^3
\varepsilon_i  g(v,e_i)e_i$, where $\varepsilon_i=g(e_i,e_i)$.
Then, $\g=\mathfrak{h} \rtimes \tilde{\mathfrak{r}}$, where
$\tilde{\mathfrak{r}}:={\rm span}(\tilde{v}=v-w)=\mathfrak{h}
^{\perp}$ (and so, $\tilde{v}$ is necessarily space-like), and $g$
takes the form (b) with respect to the pseudo-orthonormal basis
$\{e_1,e_2,e_3,e_4\}$ of $\g$, where
$e_4=\tilde{v}/\sqrt{g(\tilde{v},\tilde{v})}$.

\medskip
{\bf Case (c).} Since $g$ is Lorentzian, a subspace of $\g$ (and
so, of $\h$) on which $g$ vanishes has dimension at most one
\cite{O'N}.  Thus, being $g|_{\h}$ degenerate, its signature is
necessarily $(2,0,1)$, since all the other possibilities would
give a subspace of $\h$ dimension $\geq 2$ on which $g$ vanishes,
which cannot occur. Hence,  $\h$ admits an orthogonal basis
$\{e_1,e_2,e_3\}$, with $e_1,e_2$ unit space-like vectors and
$e_3$ a light-like vector.

If $\mathfrak{r}={\rm span} (v)$, we consider $\tilde{v}:=v-\sum
_{i=1} ^2  g(v,e_i)e_i$ and obtain $\g=\mathfrak{h} \rtimes
\tilde{\mathfrak{r}}$, with $\tilde{\mathfrak{r}}:={\rm
span}(\tilde{v})$ and $\tilde{v}$ orthogonal to $e_1,e_2$.
Moreover, because of the non-degeneracy of $g$, necessarily
$g(\tilde{v},e_3) \neq 0$.

Next, there exists a unique $\lambda_0 \in \mathbb R$, such that
$\tilde{v}+\lambda_0 e_3$ is light-like: explicitly, $\lambda_0 =
-g(\tilde{v},\tilde{v})/2g(\tilde{v},e_3)$. Putting
$k=g(\tilde{v}+\lambda _0 e_3, e_3)=g(\tilde{v}, e_3) \neq 0$ and
$e_4= \frac{1}{k}(\tilde{v}+\lambda_0 e_3)$, we get that $e_4$
acts as a derivation on $\h$, and $g$ takes the form (b) with
respect to the basis $\{e_1,e_2,e_3,e_4\}$.
\end{proof}

\noindent In the following subsections we shall classify
four-dimensional cyclic Lorentz\-ian Lie groups, treating
separately the three cases occurring in the above
Proposition~\ref{gandg}.

\subsection{First case: $\mathfrak{h}$ Riemannian}

Following \cite{M}, there exists an orthonormal basis $\{e_{1},e_{2},e_{3}\}$ of $\h$, such
that%
\begin{equation}\label{3DRie}
\lbrack e_{1},e_{2}]=a_{3}e_{3},\qquad\lbrack e_{2},e_{3}]=a_{1}e_{1}%
,\qquad\lbrack e_{3},e_{1}]=a_{2}e_{2},
\end{equation}
providing the cases listed in the following Table III, depending on the signs of $a_{1}, a_{2}$ and $a_{3}$.
\begin{center}
\begin{tabular}{|c|c|c|c|}
\hline {\em Lie group} &  $a_1$ & $a_2$ & $a_3$ $\vphantom{\displaystyle{A^{B^{C}}}}$\\
\hline  $SU(2)$ & $+$ & $+$ & $+$ $\vphantom{\displaystyle{A^{B^{C}}}}$\\
\hline  $\widetilde{SL}(2,\mathbb R)$ & $+$ & $+$ & $-$ $\vphantom{\displaystyle{A^{B^{C^{D}}}}}$\\
\hline  $\widetilde{E}(2)$ & $+$ & $+$ & $0$ $\vphantom{\displaystyle{A^{B^{C^{D}}}}}$ \\
\hline  $E(1,1)$ & $+$ & $-$ & $0$ $\vphantom{\displaystyle{A^{B^{C}}}}$\\
\hline  $H_3$ & $+$ & $0$ & $0$ $\vphantom{\displaystyle{A^{B^{C}}}}$\\
\hline  $\mathbb R^3$ & $0$ & $0$ & $0$ $\vphantom{\displaystyle{A^{B^{C}}}}$\\
\hline
\end{tabular} \nopagebreak \\ \nopagebreak Table~III: Simply connected Riemannian 3D Lie groups $\vphantom{\displaystyle\frac{a}{2}}$
\end{center}

Since $e_4$ acts as a derivation on $\h_3$, we also have
\begin{equation}\label{deriv}
\left\{
\begin{array}{l}
\lbrack e_{1},e_{4}]=c_{1}e_{1}+c_{2}e_{2}+c_{3}e_{3},\\
\lbrack e_{2},e_{4}]=p_{1}e_{1}+p_{2}e_{2}+p_{3}e_{3},\\
\lbrack e_{3},e_{4}]=q_{1}e_{1}+q_{2}e_{2}+q_{3}e_{3},
\end{array}
\right.
\end{equation}
for some constants $c_i,p_i,q_i$, which in addition must satisfy
the Jacobi identity
\begin{equation}\label{Jac}
\lbrack\lbrack
e_{i},e_{j}],e_{k}]+[[e_{j},e_{k}],e_{i}]+[[e_{k},e_{i}
],e_{j}]=0. %
\end{equation}
Applying the cyclic condition \eqref{cyclicbase} to the
pseudo-orthonormal basis satisfying \eqref{3DRie} and
\eqref{deriv}, we easily get conditions
\begin{equation}\label{cyccond1}
a_{3}+a_{1}+a_{2}=0,\quad p_{1}=c_{2}\quad q_{1}=c_{3},\quad q_{2}=p_{3}.
\end{equation}
Requiring that the Jacobi identity \eqref{Jac} holds, and after
some computations, we get the following possible solutions:
\begin{enumerate}

\item $\{a_{2}=a_{3}=0\}$. In this case, taking into account \eqref{cyccond1}, we have that $\mathfrak{h}_{3}%
=\mathrm{span}\{e_{1},e_{2},e_{3}\}=\mathbb{R}^{3}$, with the action of
$\mathbb{R}=\mathrm{span}\{e_{4}\}$ on it defined as
\begin{align}
\lbrack e_{1},e_{4}] &  =c_{1}e_{1}+p_{1}e_{2}+q_{1}e_{3}, \nonumber \\
\lbrack e_{2},e_{4}] &  =p_{1}e_{1}+p_{2}e_{2}+q_{2}e_{3}, \label{1Rie}\\
\lbrack e_{3},e_{4}] &  =q_{1}e_{1}+q_{2}e_{2}+q_{3}e_{3}. \nonumber
\end{align}

\item $\{a_{3}=-a_{2},c_{1}=p_{1}=q_{1}=0,q_{3}=p_{2}\}$. In this
case, by  \eqref{cyccond1} and the above Table~III, we conclude
that $\mathfrak{h}_{3}=\mathfrak{e}(1,1)$, with the action of
$\mathbb{R}=\mathrm{span}\{e_{4}\}$ on it defined as

\begin{equation}\label{2Rie}
\lbrack e_{1},e_{4}]   =0,\quad \lbrack e_{2},e_{4}]
=p_{2}e_{2}+q_{2}e_{3},\quad \lbrack e_{3},e_{4}]
=q_{2}e_{2}+p_{2}e_{3}.
\end{equation}

\item $\{c_{1}=p_{2},a_{3}=q_{1}=q_{2}=q_{3}=0\}$. In this case, $\mathfrak{h}_{3}=\mathfrak{e}(1,1)$, with
the action of $\mathbb{R}=\mathrm{span}\{e_{4}\}$ on it defined as%
\begin{equation}\label{3Rie}
\lbrack e_{1},e_{4}]   =p_{2}e_{1}+c_{2}e_{2},\quad
\lbrack e_{2},e_{4}]   =c_{2}e_{1}+p_{2}e_{2},\quad
\lbrack e_{3},e_{4}]   =0.
\end{equation}

\item $\{c_{1}=q_{3}, a_{2}=a_{3}=p_{1}=p_{2}=q_{2}=0\}$. This
corresponds to $\mathfrak{h}_{3}=\mathfrak{e}(1,1)$, with the
action of $\mathbb{R}=\mathrm{span}\{e_{4}\}$ on it defined as
$$
\lbrack e_{1},e_{4}]  =q_{3}e_{1}+q_{1}e_{3},\quad
\lbrack e_{2},e_{4}]  =0,\quad
\lbrack e_{3},e_{4}]   =q_{1}e_{1}+q_{3}e_{3}.
$$

\item $\{c_{1}=p_{1}=p_{2}=q_{1}=q_{2}=q_{3}=0\}$. In this case,
by the above Table~III, $\mathfrak{h}_{3}=\mathfrak{sl}(2)$ with
the trivial action of $\mathbb{R}=\mathrm{span}\{e_{4}\}$ on it.
\end{enumerate}

It is clear that the above cases (2), (3) and (4) coincide, up to
a renumeration of $e_1,e_2,e_3$. Thus, we proved the following
result.

\begin{theorem}\label{hRie}
Let $G=H \rtimes \mathbb R$ be a connected and simply connected
four-dimensional Lie group, equipped with a left-invariant
Lorentzian metric $g$, such that $g|_{H}$ is Riemannian. If $g$ is
cyclic, then the Riemannian Lie algebra $\h$ of $H$ admits an
orthonormal basis $\{e_{1},e_{2},e_{3}\}$, such that \eqref{3DRie}
holds with $a_1+a_2+a_3=0$, and one of the following cases occurs:
\begin{description}
\item[I)] $G = \mathbb R^3 \rtimes\mathbb{R}$ and the action of $\mathbb{R}=\mathrm{span}\{e_{4}\}$ (time-like) on $\mathfrak{h}%
=\mathbb{R}^{3}$ is described by \eqref{1Rie}, for arbitrary real constants $c_1,p_1,p_2,q_1,q_2,q_3$.

\vspace{4pt}\item[II)]  $G =E(1,1) \rtimes\mathbb{R}$ and the
action of $\mathbb{R}=\mathrm{span}\{e_{4}\}$ (time-like) on
$\mathfrak{h}= \mathfrak{e}(1,1)$ is described by \eqref{2Rie},
for arbitrary real constants $p_2,q_2$.

\vspace{4pt}\item[III)] $G =\widetilde{SL}(2,\mathbb R)
\times\mathbb{R}$.
\end{description}
\end{theorem}

\subsection{Second case: $\mathfrak{h}$ Lorentzian}

In this case, $\h$ is one of the unimodular Lorentzian Lie
algebras $\g_1 -\g_4$ classified in Theorem~\ref{Thom}. We treat
these cases separately.

\smallskip\noindent\textbf{1) $\h=\g_1$.} The brackets of $\g=\h \rtimes \mathfrak{r}$ are
then completely described by \eqref{g1} and \eqref{deriv}, and
the cyclic condition  \eqref{cyclicbase} gives
\[
\beta=0, \quad c_{2}=p_{1},\quad c_{3}=-q_{1},\quad p_{3}=-q_{2}.
\]
Imposing the Jacobi identity, we only have the solution%
\[
p_{1}=0,\qquad p_{2}=-q_{3},\qquad q_{1}=0,\qquad q_{2}=q_{3},
\]
so that taking into account Theorem~\ref{Thom}, we have
$\mathfrak{g=h}\rtimes\mathbb{R}$ with $\mathfrak{h}
=\mathfrak{e}(1,1)\mathfrak{=}\mathrm{span}\{e_{1},e_{2},e_{3}\}$,
$\mathbb{R}=\mathrm{span}\{e_{4}\}$ and the action given by%
\begin{equation}\label{1Lor}
\lbrack e_{1},e_{4}]   =c_{1}e_{1},\quad
\lbrack e_{2},e_{4}]   =-q_{3}(e_{2}+e_{3}),\quad
\lbrack e_{3},e_{4}]   =q_{3}(e_{2}+e_{3}).%
\end{equation}

\smallskip \noindent\textbf{2) $\h=\g_2$.}  The brackets of $\g=\h \rtimes \mathfrak{r}$
are now described by \eqref{g2} and \eqref{deriv}.  The cyclic condition  \eqref{cyclicbase} yields
\[
\alpha=-2\beta, \quad c_{2}=p_{1},\quad c_{3}=-q_{1},\quad p_{3}=-q_{2}.
\]
Finally, the Jacobi identity \eqref{Jac} admits the following two solutions:
\begin{enumerate}
\item $\{\beta=0, c_{1}=p_{1}=q_{1}=q_{2}=0\}$. Then,
$\mathfrak{g=h}\rtimes\mathbb{R}$, with $\mathfrak{h}=\mathfrak{e}
(1,1)\mathfrak{=}\mathrm{span}\{e_{1},e_{2},e_{3}\}$, $\mathbb{R}
=\mathrm{span}\{e_{4}\}$ and the action defined as
\begin{equation}\label{2Lor}
\lbrack e_{1},e_{4}]   =0,\quad \lbrack e_{2},e_{4}]
=p_{2}e_{2},\quad \lbrack e_{3},e_{4}]   =q_{3}e_{3}.
\end{equation}

\item $\{c_{1}=p_{1}=p_{2}=q_{1}=q_{2}=q_{3}=0\}$. So,
$\mathbb{R}=\mathrm{span}\{e_{4}\}$ acts trivially. Taking into
account Proposition~\ref{4DLG}, we have
$\mathfrak{g=h}\times\mathbb{R}$ with $\mathfrak{h}=\mathrm{span}
\{e_{1},e_{2},e_{3}\}=\mathfrak{sl}(2)$.
\end{enumerate}

\smallskip\noindent\textbf{3) $\h=\g_3$.}  Starting from \eqref{g3} and \eqref{deriv}, the cyclic condition \eqref{cyclicbase} now gives
\[
\alpha+\beta
+\gamma=0, \quad c_{2}=p_{1},\quad c_{3}=-q_{1},\quad p_{3}=-q_{2}.
\]
Imposing the Jacobi identity and taking into account the
classification reported in Table~I, we have the following sets of
solutions:

\begin{enumerate}

\item $\{\beta = \gamma =0\}$. This case correspond to
$\mathfrak{g=h}\rtimes\mathbb{R}$ with $\mathfrak{h}=\mathbb{R}^3$
and
\begin{align}
[ e_{1},e_{4}] &  =c_{1}e_{1}+p_{1}e_{2}-q_{1}e_{3}, \nonumber \\
[ e_{2},e_{4}] &  =p_{1}e_{1}+p_{2}e_{2}-q_{2}e_{3}, \label{1.1Lor}\\
[ e_{3},e_{4}] &  =q_{1}e_{1}+q_{2}e_{2}+q_{3}e_{3}. \nonumber
\end{align}

\item $\{\beta=0, c_{1}=q_{3},p_{1}=p_{2}=q_{2}=0\}$. Since
$\alpha+\beta +\gamma=0$, we get $\alpha+\gamma=\beta=0$. If
$\alpha=0$, we then have a special case of the previous one. For
$\alpha \neq 0$, taking into account Table~I, we have
$\mathfrak{g=h}\rtimes\mathbb{R}$, where
$\mathfrak{h}=\mathrm{span}\{e_{1},e_{2},e_{3}\}$ is
$\mathfrak{e}(2)$, $\mathbb{R}=\mathrm{span} \{e_{4}\}$ and the
action is defined as
\begin{equation}\label{3Lor}
\lbrack e_{1},e_{4}]   =q_{3}e_{1}-q_{1}e_{3},\quad \lbrack
e_{2},e_{4}]   =0,\quad \lbrack e_{3},e_{4}]
=q_{1}e_{1}+q_{3}e_{3}.
\end{equation}

\item $\{\beta=-\gamma, c_{1}=p_{1}=q_{1}=0,q_{3}=p_{2}\}$, which
is isometric to the above case, interchanging the space-like
vectors $e_1$ and $e_2$.

\item $\{\gamma=0, c_{1}=p_{2},q_{1}=q_{2}=q_{3}=0\}$. If
$\alpha=0$ we obtain a special case of case (1). When $\alpha \neq
0$, we get $\mathfrak{g=h}\rtimes\mathbb{R}$, where
$\mathfrak{h}=\mathrm{span}\{e_{1},e_{2},e_{3}\}$ is
$\mathfrak{e}(1,1)$, $\mathbb{R}=\mathrm{span} \{e_{4}\}$ and the
action is defined as
\begin{equation}\label{3.5Lor}
[ e_{1},e_{4}]   =c_{1}e_{1}+p_{1}e_{2},\quad [ e_{2},e_{4}]
=c_2e_{1}+c_1e_{2},\quad [e_{3},e_{4}] =0.
\end{equation}

\vspace{4pt}\item $\{c_{1}=p_{1}=p_{2}=q_{1}=q_{2}=q_{3}=0\}$. In
this case, the action of $e_4$ on $\h$ is trivial. Hence, by
Proposition~\ref{4DLG} and Table~I, we find that
$\mathfrak{g=h}\times\mathbb{R}$, where $\mathfrak{h}$ is either
$\mathfrak{su}(2)$ or $\mathfrak{sl}(2)$.
\end{enumerate}

\smallskip\noindent\textbf{4): $\h=\g_4$.} By \eqref{g4} and \eqref{deriv},
the cyclic condition holds if and only if
\[
\alpha=2(\varepsilon-\beta), \quad c_{2}=p_{1},\quad c_{3}=-q_{1},\quad p_{3}=-q_{2}.
\]
Then, imposing the Jacobi identity and taking into account
Proposition~\ref{4DLG}, we have the following two non-isometric
cases:
\begin{enumerate}
\item $\{\beta=\varepsilon, c_{1}=0,p_{1}=\varepsilon
q_{1},q_{2}=\frac{\varepsilon }{2}\left(  p_{2}-q_{3}\right)\}$.
In this case, $\mathfrak{g=h}\rtimes\mathbb{R}$, where
$\mathfrak{h}=\mathfrak{n}_{3} =\mathrm{span}\{e_{1},e_{2}
,e_{3}\}$ is the Heisenberg Lie algebra,
$\mathbb{R}=\mathrm{span}\{e_{4}\}$ and the action is defined as
\begin{align}\label{4Lor}
\lbrack e_{1},e_{4}]  &=q_{1}(e_{2}-e_{3}), \nonumber \\
\lbrack e_{2},e_{4}]  &=q_{1}e_{1}+p_{2}e_{2}-q_{2}e_{3},\\
\lbrack e_{3},e_{4}]  &=q_{1}e_{1}+q_{2}e_{2}+q_{3}e_{3}, \nonumber
\end{align}
with $q_{2}=\frac{\varepsilon}{2}\left(  p_{2}-q_{3}\right)$.

\vspace{4pt}\item $\{c_{1}=p_{1}=p_{2}=q_{1}=q_{2}=q_{3}=0\}$, so
that $\mathfrak{g=h}\times\mathbb{R}$ trivially, and, taking into
account Proposition~\ref{4DLG},  $\h =
\mathfrak{sl}(2)$.
\end{enumerate}

Collecting all the above cases, we obtain the following.

\begin{theorem}\label{hLore}
Let $G=H \rtimes \mathbb R$ be a connected and simply connected
four-dimensional Lie group, equipped with a left-invariant
Lorentzian metric $g$, such that $g|_{H}$ is Lorentzian. If $g$ is
cyclic, then the Lorentzian Lie algebra $\h$ of $H$ admits a
pseudo-orthonormal basis $\{e_{1},e_{2},e_{3}\}$, with $e_3$
time-like, such that one of the following cases occurs:
\begin{description}
\item[I)] $G = E(1,1) \rtimes \mathbb R$  and one of the following holds:
\begin{itemize}
\item[(a)]
$\mathfrak{e}(1,1)\mathfrak{=}\mathrm{span}\{e_{1},e_{2},e_{3}\}$
is of the form $\g_1$ with $\beta=0$, and the action of
$\mathbb{R}=\mathrm{span}\{e_{4}\}$ on $\mathfrak{e}(1,1)$ is
described by \eqref{1Lor}.

\item[(b)]
$\mathfrak{e}(1,1)\mathfrak{=}\mathrm{span}\{e_{1},e_{2},e_{3}\}$
is of the form $\g_2$ with $\alpha=\beta=0$, and the action of
$\mathbb{R}=\mathrm{span}\{e_{4}\}$ on $\mathfrak{e}(1,1)$ is
described by \eqref{2Lor}.

\item[(c)]
$\mathfrak{e}(1,1)\mathfrak{=}\mathrm{span}\{e_{1},e_{2},e_{3}\}$
is of the form $\g_3$ with $\gamma=0$, and the action of
$\mathbb{R}=\mathrm{span}\{e_{4}\}$ on $\mathfrak{e}(1,1)$ is
described by \eqref{3.5Lor}.

\end{itemize}

\item[II)] $G = \widetilde{SL}(2,\mathbb R) \times \mathbb R$,
with $\mathbb{R}=\mathrm{span}\{e_{4}\}$ acting trivially  on
$\mathfrak{sl}(2)$
$\mathfrak{=}\mathrm{span}\{e_{1},e_{2},e_{3}\}$, and one of the
following holds:
\begin{itemize}
\vspace{4pt}\item[(a)] $\mathfrak{sl}(2)$ is of the form $\g_2$ with $\alpha=-2\beta \neq 0$.
\vspace{4pt}\item[(b)]$\mathfrak{sl}(2)$ is of the form $\g_3$ with $\alpha+\beta+\gamma=0$.
\vspace{4pt}\item[(c)]$\mathfrak{sl}(2)$ is of the form $\g_4$ with $\alpha=2(\varepsilon-\beta) \neq 0$.
\end{itemize}

\vspace{4pt}\item[III)] $G = \tilde{E}(2) \rtimes \mathbb R$,
where $\mathfrak{e}(2)=\mathrm{span}\{e_{1},e_{2},e_{3}\}$ is of
the form $\g_3$ with $\alpha+\gamma=\beta=0$, and the action of
$\mathbb{R}=\mathrm{span}\{e_{4}\}$ on $\mathfrak{e}(2)$ is
described by \eqref{3Lor}.

\vspace{4pt}\item[IV)] $G = \mathbb R^3 \rtimes \mathbb R$, where
$\mathbb R^3 \mathfrak{=}\mathrm{span}\{e_{1},e_{2},e_{3}\}$ and
the action of $\mathbb{R}=\mathrm{span}\{e_{4}\}$ on $\mathbb R^3$
is described by \eqref{1.1Lor}.

\vspace{4pt}\item[V)] $G = SU(2) \times \mathbb R$, where
$\mathfrak{su}(2)\mathfrak{=}\mathrm{span}\{e_{1},e_{2},e_{3}\}$
is of the form $\g_3$ with $\alpha+\beta+\gamma=0$, and the action
of $\mathbb{R}=\mathrm{span}\{e_{4}\}$ on $\mathfrak{su}(2)$ is
trivial.

\vspace{4pt}\item[VI)] $G = H_3 \rtimes \mathbb R$, where
$\mathfrak{n}_3 \mathfrak{=}\mathrm{span}\{e_{1},e_{2},e_{3}\}$ is
of the form $\g_4$ with $\alpha=\beta-\varepsilon=0$, and the
action of $\mathbb{R}=\mathrm{span}\{e_{4}\}$ on $\mathfrak{n}_3$
is described by~\eqref{4Lor}.

\end{description}
\end{theorem}

\subsection{Third case: $\mathfrak{h}$ degenerate}

We now assume that the restriction of the metric $g$ on
$\mathfrak{h}$ is degenerate. It is enough to restrict to the case
when the derived algebra is the full subalgebra $\mathfrak{h}$,
that is,
\[
\g' =\lbrack \mathfrak{g},\mathfrak{g}]=\mathfrak{h}.
\]
In fact, if dim$\g '<3$, then there are at least two linearly
independent vectors acting as derivations in $\mathfrak{g}$. Since
$\mathfrak{g}$ is Lorenztian, the subspace spanned by these two
vectors cannot be completely null \cite{O'N} and so, we can pick a
derivation that is either space-like or time-like. Henceforth, we
are in one of the non-degenerate situations already studied in the
previous subsections.

We shall now investigate the different possibilities, compatible
with condition $\g '=\mathfrak{h}$, determined by the dimension of
the derived algebra
$\mathfrak{h}^{\prime}=[\mathfrak{h},\mathfrak{h}]$ of $\h$.

\medskip\noindent
{\bf dim$\mathfrak{h}^{\prime}=0$.}  In this case, $\h=\mathbb
R^3$ is abelian. As the only non-vanishing  Lie brackets are given
by \eqref{deriv} and $\h=\g '$ is abelian, the Jacobi identity
holds trivially. Moreover, the metric $g$ is cyclic if and only if
$c_2=p_1, q_1=q_2=0$. Therefore, the Lie algebra is completely
described by

\begin{equation}\label{0deg}
\lbrack e_{1},e_{4}]=c_{1}e_{1}+p_{1}e_{2}+c_{3}e_{3},\; \lbrack e_{2},e_{4}]=p_{1}e_{1}+p_{2}e_{2}+p_{3}e_{3},\;
\lbrack e_{3},e_{4}]=q_{3}e_{3}.
\end{equation}

\medskip\noindent
{\bf dim$\mathfrak{h}^{\prime}=1$.}  Then,
$\mathfrak{h}=\mathfrak{n}_3$ is the three-dimensional Heisenberg
Lie algebra and so, $\mathfrak{h}^{\prime}= {\rm span}(X)$.

As it follows from case (c) in Proposition~\ref{4DLG}, $g|_{\h}$
has signature $(2,0,1)$. Thus, we can write $X=V+\lambda e_{3}$,
where $V$ is spacelike and $e_{3} \perp V$ is null. We have the
following two possibilities.

\smallskip{\bf (a): $V\neq0$.}

We consider $e_{1}=X/\left\Vert X\right\Vert $ (space-like) and
complete the basis of $\mathfrak{h}$ with another space-like unit
vector $e_{2}$ and the null vector $e_{3}$, so that
\[
g|_{\mathfrak{g}_{3}}=\left(
\begin{array}
[c]{ccc}%
1 & 0 & 0\\
0 & 1 & 0\\
0 & 0 & 0
\end{array}
\right)  ,\qquad\left\{
\begin{array}
[c]{c}
\lbrack e_{1},e_{2}]=\alpha e_{1},\\
\lbrack e_{1},e_{3}]=\beta e_{1},\\
\lbrack e_{2},e_{3}]=\mu e_{1}.
\end{array}
\right.
\]
Imposing the cyclic condition, we get
\begin{equation}\label{eq1}
\mu=0,\quad c_{2}=p_{1},\quad q_{1}=0,\quad q_{2}=0.
\end{equation}
Next, we apply the Jacobi identity \eqref{Jac} and find the
following four possible solutions:

\begin{itemize}
\item $\{c_{3}=p_{1}=q_{3}=0,p_2\alpha=-p_{3}\beta \}$. Taking
into account \eqref{eq1}, We have
\begin{equation}\label{yy}
\begin{array}{llll}
\lbrack e_{1},e_{2}] =\alpha e_{1}, & \lbrack e_{1},e_{3}] =\beta e_{1}, & \lbrack e_{2},e_{3}] =0, &  \\[4pt]
\lbrack e_{1},e_{4}]=c_{1}e_{1}, & \lbrack e_{3},e_{4}]=0, & \lbrack e_{2},e_{4}]=p_{2}e_{2}+p_{3}e_{3}, & p_2\alpha+p_{3}\beta =0.
\end{array}
\end{equation}

\item $\{\alpha=\beta=0\}$. But since $\mu=0$ by \eqref{eq1}, this
case would contradict dim$\mathfrak{h}^{\prime}=1$ and so, it does
not occur.

\item $\{c_{3}=p_{1}=p_{2}=p_{3}=q_{3}=0\}$. Then, by \eqref{eq1},
we would conclude that $\dim[\mathfrak{g,g]}<3$, against our
assumption.

\item $\{\beta=c_{3}=p_{1}=p_{2}=0\}$, which, taking into account
\eqref{eq1}, contradicts again $\dim[\mathfrak{g,g]}=3$.
\end{itemize}

\smallskip{\bf (b): $V=0$.}

We can then choose an orthogonal basis $\{e_{1},e_{2},e_{3}\}$ of
$\h$, such that
\[
g|_{\mathfrak{g}_{3}}=\left(
\begin{array}
[c]{ccc}%
1 & 0 & 0\\
0 & 1 & 0\\
0 & 0 & 0
\end{array}
\right)  ,\qquad\left\{
\begin{array}
[c]{c}%
\lbrack e_{1},e_{2}]=\alpha e_{3},\\
\lbrack e_{1},e_{3}]=\beta e_{3},\\
\lbrack e_{2},e_{3}]=\mu e_{3}.%
\end{array}
\right.
\]
Imposing the cyclic condition we get%
\[
\mu=0, \quad c_{2}=p_{1}+\alpha, \quad q_{2}=\mu=0, \quad q_{1}=-\beta
\]
and applying the Jacobi identity we have the following possible solutions:

\begin{itemize}
\item $\{\beta=0, c_{1}=-p_{2}+q_{3}\}$. Then, we have%
\begin{equation}\label{yyy}
\begin{array}{ll}
\lbrack e_{1},e_{2}]=\alpha e_{3}, \quad & \lbrack e_{1},e_{4}]=(q_{3}-p_{2})e_{1}+(p_{1}+\alpha)e_{2}+c_{3}e_{3},\\
\lbrack e_{1},e_{3}]=0, \quad & \lbrack e_{2},e_{4}]=p_{1}e_{1}+p_{2}e_{2}+p_{3}e_{3},\\
\lbrack e_{2},e_{3}]=0, \quad &  \lbrack e_{3},e_{4}]=q_{3}e_{3}.%
\end{array}
\end{equation}
\item $\{\alpha=\beta=0\}$. But since $\mu=0$, this contradicts
dim$\mathfrak{h}^{\prime}=1$ and so, it cannot occur.

\end{itemize}

\medskip\noindent
{\bf dim$\mathfrak{h}^{\prime}=2$.} Thus, either
$\mathfrak{h}=\mathfrak{e} (1,1)$ or
$\mathfrak{h}=\mathfrak{e}(2)$.

Taking into account the signature of $g|_{\h}$ as in the previous
case, we now have
$\mathfrak{h}^{\prime}=\mathrm{span}\{X_{1},X_{2}\}$, where
$X_{i}=V_{i}+\lambda_{i}e_{3}$, with $V_{i}$ space-like and
$e_{3}$ null and orthogonal to $V_1,V_2$. We consider the
following subcases.

\smallskip{\bf (a): $V_{1}$ and $V_{2}$ are linearly independent.}

Since $V_{1},V_2$ are space-like, there exist orthonormal vectors $e_{1}$ and $e_{2}$, such that $\h^\prime =\mathrm{span}%
\{X_{1},X_{2}\}=\mathrm{span}\{e_{1},e_{2}\}$. With respect to the basis $\{e_{1},e_{2},e_{3}\}$ of $\h$, we then have%
\[
g|_{\mathfrak{g}_{3}}=\left(
\begin{array}
[c]{ccc}%
1 & 0 & 0\\
0 & 1 & 0\\
0 & 0 & 0
\end{array}
\right)  ,\qquad\left\{
\begin{array}{l}%
\lbrack e_{1},e_{2}]=a_{1}e_{1}+a_{2}e_{2},\\
\lbrack e_{1},e_{3}]=b_{1}e_{1}+b_{2}e_{2},\\
\lbrack e_{2},e_{3}]=t_{1}e_{1}+t_{2}e_{2}.%
\end{array}
\right.
\]
Imposing the cyclic condition, we find%
\[
b_{2}=t_{1}, \quad c_{2}=p_{1}, \quad q_{1}=0, \quad q_{2}=0.
\]
However, when we apply the Jacobi identity, all the solutions we
get turn out to be incompatible with either dim$\g'=3$ or dim$\h
'=2$. For example, one of such solutions is given by
$$\{b_{1}=0,c_{1} a_{2}^{2}=p_{2}a_{1}^{2},p_{1}a_2=p_{2}a_1,p_{3}a_2=-a_{1}c_{3},t_{1}=0,t_{2}=0\}.$$
But then, $\lbrack e_{1},e_{3}]=\lbrack e_{2},e_{3}]=0$,
contradicting the fact that dim$\mathfrak{h}^{\prime}=2$. So, this
case does not occur.

\smallskip{\bf (b): $V_{1}$ and $V_{2}$ are linearly dependent.}

Then, we can choose $\{V_1,e_{3}\}$ as a basis for
$\mathfrak{h}^{\prime}$. We  consider $e_{1}=V_1/\left\Vert
V_1\right\Vert $, and a
space-like vector $e_{2}$, orthogonal to both  $e_{1}$ and $e_{3}$, so that we have%
\[
g|_{\mathfrak{g}_{3}}=\left(
\begin{array}
[c]{ccc}%
1 & 0 & 0\\
0 & 1 & 0\\
0 & 0 & 0
\end{array}
\right)  ,\qquad\left\{
\begin{array}{l}
\lbrack e_{1},e_{2}]=a_{1}e_{1}+ a_{3}e_{3},\\
\lbrack e_{1},e_{3}]=b_{1}e_{1}+ b_{3}e_{3},\\
\lbrack e_{2},e_{3}]=t_{1}e_{1}+ t_{3}e_{3}.
\end{array}
\right.
\]
Imposing the cyclic condition, we get
\[
t_{1}=0, \quad c_2= a_{3}+p_{1}, \quad q_1=-b_{3}, \quad
q_{2}=-t_{3}.
\]
Also in this case, the Jacobi identity does not provide any
solutions compatible with dim$\g^\prime=3$ and dim$\h^\prime =2$.
Therefore, this case cannot occur.

\medskip\noindent
{\bf dim$\mathfrak{h}^{\prime}=3$.}

As above, we consider that $e_{3}\in\mathfrak{h}$ is orthogonal to
$\mathfrak{h}$ itself. Since $\mathfrak{h}^{\prime}=\mathfrak{h}$,
we have either $\mathfrak{h}=\mathfrak{sl}(2)$ or
$\mathfrak{h}=\mathfrak{su}(2)$. In order to distinguish these two
cases, we consider
$\mathrm{ad}_{e_{3}}:\mathfrak{h}\rightarrow\mathfrak{h}$, which,
since $\mathfrak{h}^{\prime}=\mathfrak{h}$, is necessarily of rank
$2$. Besides $0$, $\mathrm{ad}_{e_{3}}$ has either two real
eigenvalues or two conjugate complex eigenvalues. In addition, if
we write $e_{3}=[X_{1},X_{2}]$,
we have%
\[
\mathrm{ad}_{e_{3}}=\mathrm{ad}_{X_{1}}\circ\mathrm{ad}_{X_{2}}-\mathrm{ad}%
_{X_{2}}\circ\mathrm{ad}_{X_{1}}
\]
so that $\mathrm{tr}(\mathrm{ad}_{e_{3}})=0$. We thus have the
following possible cases.

\smallskip{\bf (a): Eigenvalues of $\mathrm{ad}_{e_{3}}$ are $0,\lambda \neq 0$ and
$-\lambda$.}

We choose $e_{1}$ and $e_{2}$ (unitary) eigenvectors, that is,
$[e_{3} ,e_{1}]=\lambda e_{1}$, $[e_{3},e_{2}]=-\lambda e_{2}$.
The Jacobi identity (rescaling $e_{3}$ if needed) gives
$[e_{2},e_{1}]=e_{3}$. With respect to $\{e_1,e_2,e_3\}$, the
metric is given by
\[
g|_{\mathfrak{h}}=\left(
\begin{array}
[c]{ccc}
1 & k & 0\\
k & 1 & 0\\
0 & 0 & 0
\end{array}
\right)  .
\]
Imposing the cyclic condition, we then find
$$\left\{\begin{array}{l}
2k\lambda  =0,\\
q_{1}+kq_{2}   =0,\\
kq_{1}+q_{2}   =0,\\
1+kc_{1}-p_{1}+c_{2}-kp_{2}  =0,
\end{array}\right.
$$
which, since $\lambda \neq 0$, easily reduces to $k=q_{1}=q_{2}=0,
p_{1}=1+c_{2}$.

Imposing the Jacobi identity to $\mathfrak{g}$, we get
$c_{1}=-p_{2}+q_{3}$ and $\lambda=0$, which is a contradiction.
Hence, this case cannot occur.

\smallskip{\bf (b): Eigenvalues of $\mathrm{ad}_{e_{3}}$ are $0,i\beta$ and
$-i\beta$, with $\beta \neq 0 $.}

We choose $e_{1}$ and $e_{2}$ (unitary) Jordan vectors, that is,
$[e_{3} ,e_{1}]=\beta e_{2}$, $[e_{3},e_{2}]=-\beta e_{1}$. The
Jacobi identity (rescaling $e_{3}$ if needed) then gives
$[e_{1},e_{2}]=\beta e_{3}$, and the metric is described by
\[
g|_{\mathfrak{h}}=\left(
\begin{array}
[c]{ccc}%
1 & k & 0\\
k & 1 & 0\\
0 & 0 & 0
\end{array}
\right)  .
\]
Imposing the cyclic condition for $e_{1}$, $e_{2}$ and $e_{3}$, we have
\[
0=g([e_{1},e_{2}],e_{3})+g([e_{2},e_{3}],e_{1})+g([e_{3},e_{1}],e_{2})=2\beta
\]
which is not admissible. Therefore, this case does not occur.

Collecting all the above cases, we obtain the following.

\begin{theorem}\label{hDeg}
Let $G=H \rtimes \mathbb R$ be a connected, simply connected
four-dimensional Lie group, equipped with a left-invariant
Lorentzian metric $g$, such that $g|_{H}$ is degenerate. If $g$ is
cyclic, then we can choose a basis $\{e_1,e_2,e_3,e_4\}$ of the
Lie algebra $\g = \h \rtimes \mathbb{R}$, such $\h =
\mathrm{span}(e_1,e_2,e_3)$, with respect to $\{e_i\}$ the the
metric is described as in case~(c) of
Proposition~{\em\ref{gandg}}, and one of
 the following cases occurs:
 \begin{description}
\item[I)] $G = \mathbb{R}^3 \rtimes \mathbb R$, with brackets as in
\eqref{0deg}.

\vspace{4pt}\item[II)] $G=H_3\rtimes \mathbb{R}$ with brackets either as in
\eqref{yy} or as in \eqref{yyy}.
\end{description}
\end{theorem}

\section{Final remarks}

\subsection{Homogeneous manifolds with homogeneous structures in $\mathcal{S}_3$ and in $\mathcal{S}_1 \oplus
\mathcal{S}_2$.}

We consider the question whether a homogeneous manifold can admit
homogeneous structures both in $\mathcal{S}_3$ and in
$\mathcal{S}_1 \oplus \mathcal{S}_2$.

If we require that {\em the same} homogeneous structure $S$
belongs to both $\mathcal{S}_3$ and $\mathcal{S}_1 \oplus
\mathcal{S}_2$, then it means that $S=0$, that is, the manifold is
symmetric, and conversely.

Observe that for a metric Lie group $G$, equipped with a
left-invariant pseudo-Riemannian metric $g$, we are considering a
specific homogeneous structure $\tilde S$, namely, the one giving
to it the Lie group structure ($G$ acting transitively on itself
by isometries). Thus, the fact that such a structure belongs to
both $\mathcal{S}_3$ and $\mathcal{S}_1\oplus \mathcal{S}_2 $ is
equivalent to require  that $(G,g)$ is a symmetric Lie group.

On the other hand, for example, it follows from
Theorem~\ref{hLore} that the homogeneous structure $\tilde S$ of
$SU(2) \times \mathbb R$ belongs to $\mathcal{S}_1 \oplus
\mathcal{S}_2$, since the left-invariant metric $g$ is cyclic. At
the same time, $SU(2)$ is a non-symmetric naturally reductive
homogeneous Lorentzian manifold \cite[Theorem~4.3]{CM}.
Consequently, being the (non-symmetric) direct product of
naturally reductive manifolds, four-dimensional Lorentzian Lie
group $SU(2) \times \mathbb R$ also admits a (non-trivial)
homogeneous structure $S \in \mathcal{S}_3 $.

\subsection{Three-dimensional cotorsionless Lorentzian manifolds.}

We already recalled in Section 3 that all connected, simply
connected homogeneous Lorentzian three-manifolds can be realized
as Lorentzian Lie groups, with the only exception of $\mathbb S^2
\times \mathbb R$ with the product metric $g=g_{\mathbb
S^2}-dt^2$. It is obvious that as a product of symmetric spaces,
$\mathbb S^2 \times \mathbb R$ is again symmetric and so, it is
(trivially) a cotorsionless manifold. With regard to all
homogeneous structures on $\mathbb S^2 \times \mathbb R$, it is
possible to check by direct calculation that they are parametrized
by one parameter, and the only tensor belonging to $\mathcal{S}_1
+ \mathcal{S}_2$ is $S=0$.

The next result then follows from the above observations about
$\mathbb S^2 \times \mathbb R$ and  the classification of
three-dimensional cyclic Lorentzian Lie groups given in
Theorem~\ref{cyclic3D}.

\begin{theorem}
A three-dimensional connected, simply connected cotorsionless
homogeneous Lorentzian manifold is either isometric to $\mathbb
S^2 \times \mathbb R$, or to one of the cyclic Lorentzian Lie
groups classified in Theorem~{\em\ref{cyclic3D}}.
\end{theorem}

\subsection{Relating three- and four-dimensional cyclic  Lie groups}
As proved in \cite{GGO}, a three-dimensional Riemannian Lie group
$H$ is cyclic if and only if its Lie algebra is of the form
\eqref{3DRie} with $a_1+a_2+a_3$. Moreover, by direct calculation
(see also the proof of Theorem~6.2 in \cite{GGO}), we see that if
$(G=H \rtimes \mathbb R,g)$ is a four-dimensional cyclic
Riemannian Liegroup, then $(H,g_{\h})$ is again cyclic.

With regard to cyclic Lorentzian metrics, by Theorem~\ref{hRie} we
see that if $(G=H \rtimes \mathbb R,g)$ is a four-dimensional
cyclic Lorentzian Lie group, with $H$ Riemannian, then
$(H,g_{\h})$ is cyclic. Similarly, Theorems~\ref{cyclic3D} and
\ref{hLore} show that if $(G=H \rtimes \mathbb R,g)$ is a
four-dimensional cyclic Lorentzian Lie group and $H$ is
Lorentzian, then $(H,g_{\h})$ is cyclic.

Hence, when $g_{\h}$ is either Riemannian or Lorentzian,
left-invariant cyclic Lorentzian metrics on four-dimensional Lie
groups can be interpreted as semi-direct product extensions of
corresponding cyclic metrics on three-dimensional Lie algebras.
But clearly, the examples listed in Theorem~\ref{hDeg} do not show
such a correspondence, because for them $g_{\h}$ is degenerate.

So, we see once more that geometric behaviours occurring in
Lorentzian settings are richer that their Riemannian analogues:
four-dimensional Riemannian cyclic metrics are semi-direct product
extensions of three-dimension\-al Riemannian cyclic metrics, while
not all four-dimensional Lorentzian cyclic metrics arise from a
corresponding construction.

\subsection{Compact homogeneous solvmanifolds from cyclic Lie
groups}

From the classification results obtained in Section~4, all
four-dimensional simply connected Lorentzian cyclic Lie groups $G$
are non compact. One could ask about the existence compact
Lorentzian cotorsionless manifolds by considering quotients
$G/\Gamma$ by an appropriate lattice subgroup $\Gamma\subset G$.
This is precisely the way compact homogeneous solvmanifolds or
nilmanifolds are constructed. However, the following results holds
(for arbitrary dimension of $G$).

\begin{proposition}
Let $M=G/\Gamma$ be a compact pseudo-Riemannian homogeneous
solvmanifold (in particular, a nilmanifold) given by the quotient
of the right action of a lattice $\Gamma$ in a solvable (in
particular, nilpotent) Lie group $G$. We assume that $G$ is
equipped with a left-invariant metric $g$\ such that the
projection $\pi:G\rightarrow M$ is a local isometry. Then, the
metric $g$ is also right-invariant, the group is naturally
reductive and the homogeneous structure associated to $g$ belongs
to of class $\mathcal{S}_{3}$.

Consequently, the only possible cyclic homogeneous structure for
$G$ is the trivial one and occurs when $M$ is locally symmetric.
\end{proposition}

\begin{proof}
The bi-invariance follows from the classification of homogeneous
compact Lorentzian spaces obtained in \cite{Ze}. From here, the
only cyclic homogeneous structure is the trivial one and hence $G$
(and $M$) is locally symmetric.
\end{proof}

\end{document}